\documentclass[12pt, a4paper]{article}

\usepackage{a4wide}
\usepackage{graphicx}
\usepackage[figurename=Fig.]{caption}
\usepackage{amsthm}
\usepackage{amsmath}

\theoremstyle{plain}
\newtheorem{theorem}{Theorem}
\newtheorem*{lemma}{Lemma}

\def \t {\triangle}
\def \e {\varepsilon}
\def \o {\omega}

\begin{document}

\title{The Mixed Poncelet--Steiner Closure Theorem}
\author{Nikolai Ivanov Beluhov}
\date{}

\maketitle

\begin{abstract}

We state and prove a new closure theorem closely related to the classical closure theorems of Poncelet and Steiner. Along the way, we establish a number of theorems concerning conic sections.
	
\end{abstract}

\section{A Sangaku Problem}

\begin{theorem}
\label{t1}

Given are two circles $k$ and $\o$ such that $\o$ lies inside $k$. Let $t$ be any line tangent to $\o$. Let $\o_1$ and $\o_2$ be the two circles which are tangent externally to $\o$, internally to $k$, are tangent to $t$, and lie on the same side of $t$ as $\o$. If the circles $\o$, $\o_1$, and $\o_2$ have a second common external tangent $u$ for some position of the line $t$, then they have a second common external tangent $u$ for any position of the line $t$. (Fig.\ \ref{f1})

\end{theorem}

\begin{figure}[ht]\begin{center}
\includegraphics[width=0.55\textwidth]{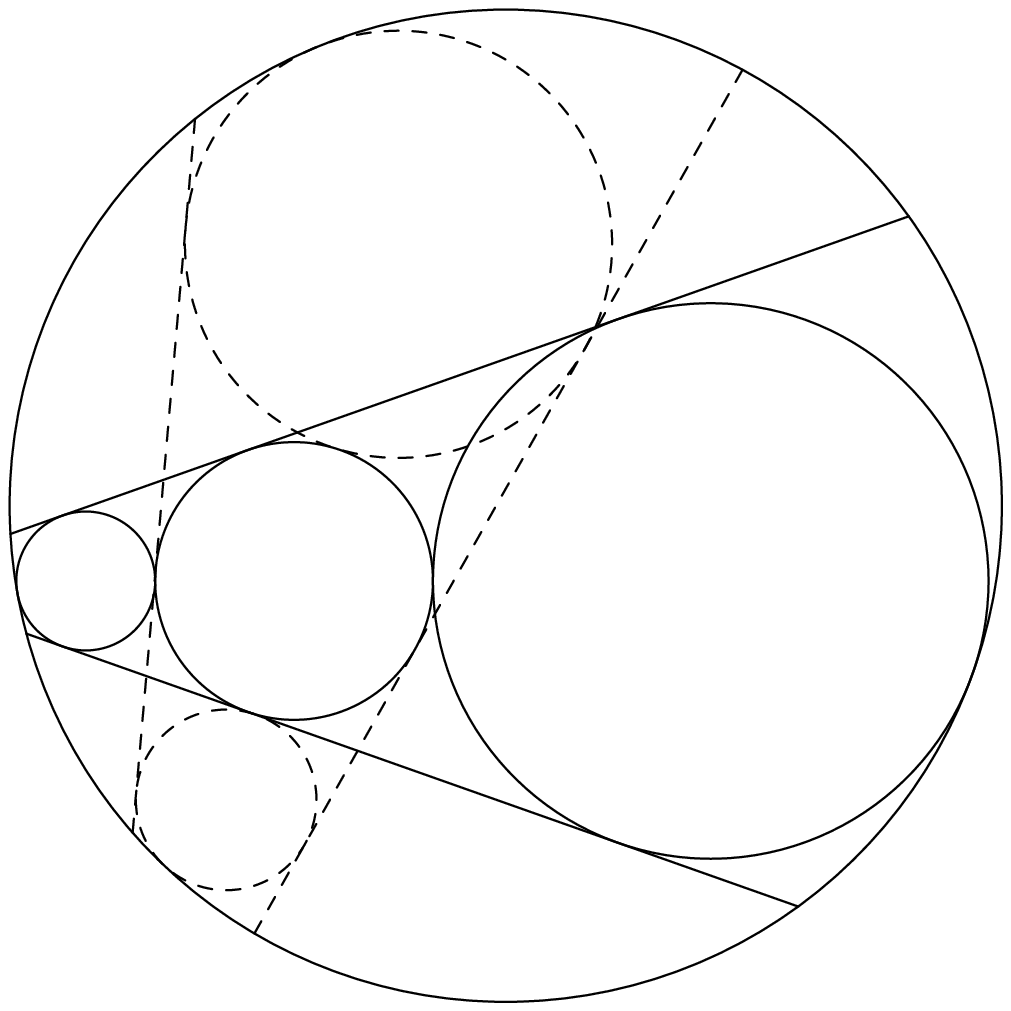}
\caption{}
\label{f1}
\end{center}\end{figure}

The author discovered this fact while contemplating an 1839 Sangaku problem from a tablet found in Nagano Prefecture \cite{fp89}. The original problem asks for a proof that the product of the radii of the largest circles inscribed in the circular segments which $t$ and $u$ cut from $k$ in Fig.\ \ref{f1} equals the square of the radius of $\omega$.

Theorem \ref{t1} admits a straightforward algebraic proof.

\begin{proof}[Proof of Theorem \ref{t1}]

Let $I$ and $r$ be the center and radius of $\o$ and $O$ and $R$ be the center and radius of $k$.

Let $\o_1$ and $\o_2$ be two circles inscribed in the annulus between $\o$ and $k$. Let $r_1$ and $r_2$ be their radii, $A$ and $B$ be their contact points with $\o$, and $C$ and $D$ be their contact points with $k$. Then the circles $\o$, $\o_1$, and $\o_2$ have two common external tangents exactly when $AB$ is a diameter of $\o$ and $r_1r_2 = r^2$.

Let $S$ be the internal similitude center of $\o$ and $k$; by the three similitude centers theorem, $S = AC \cap BD$. Let the lines $AS$ and $BS$ intersect $\o$ for the second time in $P$ and $Q$.

In this configuration, let $A$ and $B$ vary so that $AB$ is a diameter of $\o$. We wish to establish when $r_1r_2 = r^2$.

Let $s^2$ be the power of $S$ with respect to $\o$, $s^2 = SA \cdot SP = SB \cdot SQ$, and $m^2 = \frac{1}{2}(SA^2 + SB^2)$. Then both $s^2$ and $m^2$ remain constant as $AB$ varies: this is obvious for $s^2$, and for $m^2$ it follows from the fact that the side $AB$ and the median $SI$ to this side in $\t SAB$ remain constant.

Let $SA^2 = m^2 + x^2$ and $SB^2 = m^2 - x^2$. Notice that, when $AB$ varies, $x^2$ varies also.

We have $r : r_1 = AP : AC$, $r : r_2 = BQ : BD$, and $r : R = SP : SC = SQ : SD$. Therefore, \[ r_1 = \frac{s^2R - m^2r - x^2r}{s^2 + m^2 + x^2} \textrm{ and } r_2 = \frac{s^2R - m^2r + x^2r}{s^2 + m^2 - x^2}. \] It follows that \[ \begin{split} r_1r_2 = r^2 &\Leftrightarrow (s^2R - m^2r)^2 - x^4r^2 = r^2[(s^2 + m^2)^2 - x^4]\\ &\Leftrightarrow (s^2R - m^2r)^2 = r^2(s^2 + m^2)^2, \end{split} \] which does not depend on $x$. Therefore, if $\o$, $\o_1$, and $\o_2$ have two common external tangents for some position of $AB$, then they do so for any position of $AB$, as needed. \qedhere 

\end{proof}

As a corollary we obtain a simple characterization of the pairs of circles satisfying the condition of Theorem \ref{t1} which is strongly reminiscent of Euler's formula \cite{euf}. Namely, two circles $k$ and $\o$ of radii $R$ and $r$ and intercenter distance $d$ satisfy the condition of Theorem \ref{t1} exactly when \[ d^2 = (R - r)^2 - 4r^2. \]

An additional curious property of the figure is as follows.

\begin{theorem}
\label{t2}

In the configuration of Theorem \ref{t1}, the intersection $t \cap u$ of the common external tangents of $\o$, $\o_1$, and $\o_2$ describes a straight line when $t$ varies.

\end{theorem}

\begin{proof}

When $I \equiv O$, the line in question is the line at infinity. Let $I \not\equiv O$ and let the line $IO$ intersect $k$ in $A$ and $D$ and $\o$ in $B$ and $C$ so that $B$ and $C$ lie on the segment $AD$ in this order and $AB < CD$. Consider a coordinate system in which the coordinates of $A$, $B$, $C$, and $D$ are $(1, 0)$, $(a, 0)$, $(a^2, 0)$, and $(a^3, 0)$, respectively, for some real $a > 1$.

Consider an arbitrary circle $\o'$ of center $J$ of coordinates $(x, y)$ and radius $r'$ inscribed in the annulus between $\o$ and $k$. We have \[ IJ = r + r' \Rightarrow \left (\frac{a^2 + a}{2} - x \right )^2 + y^2 = \left (\frac{a^2 - a}{2} + r' \right )^2 \] and \[ OJ = R - r' \Rightarrow \left (\frac{a^3 + 1}{2} - x \right )^2 + y^2 = \left (\frac{a^3 - 1}{2} - r' \right )^2. \]

From these, we obtain \[ \left (\frac{a^2 + a}{2} - x \right )^2 - \left (\frac{a^3 + 1}{2} - x \right )^2 = \left (\frac{a^2 - a}{2} + r' \right )^2 - \left (\frac{a^3 - 1}{2} - r' \right )^2 \Rightarrow r' = \frac{a - 1}{a + 1} \cdot x. \]

Therefore, the ratio $r : r'$ equals the ratio of the distances from $I$ and $J$ to the $y$-axis $l$, implying that the common external tangents of $\o$ and $\o'$ intersect on $l$ as $\o'$ varies, as needed. \qedhere

\end{proof}

\section{The Mixed Poncelet--Steiner Closure Theorem}

Theorem \ref{t1} admits a powerful generalization.

\begin{theorem}[The Mixed Poncelet--Steiner Closure Theorem]
\label{t3}

Given are two circles $k$ and $\o$ such that $\o$ lies inside $k$. Let $\o_1$ be any circle inscribed in the annulus between $k$ and $\o$. Let $t_1$ be a common external tangent of $\o_1$ and $\o$. Let $\o_2$ be the second circle which is inscribed in the annulus between $k$ and $\o$, is tangent to $t_1$, and lies on the same side of $t_1$ as $\o$. Let $t_2$ be the second common external tangent of $\o_2$ and $\o$. Let $\o_3$ be the second circle which is inscribed in the annulus between $k$ and $\o$, is tangent to $t_2$, and lies on the same side of $t_2$ as $\o$, etc. If $\o_{n + 1} \equiv \o_1$ for some choice of the initial circle $\o_1$ and the initial tangent $t_1$, then $\o_{n + 1} \equiv \o_1$ for any choice of the initial circle $\o_1$ and the initial tangent $t_1$. (Fig.\ \ref{f2})

\end{theorem}

\begin{figure}[ht]\begin{center}
\includegraphics[width=0.55\textwidth]{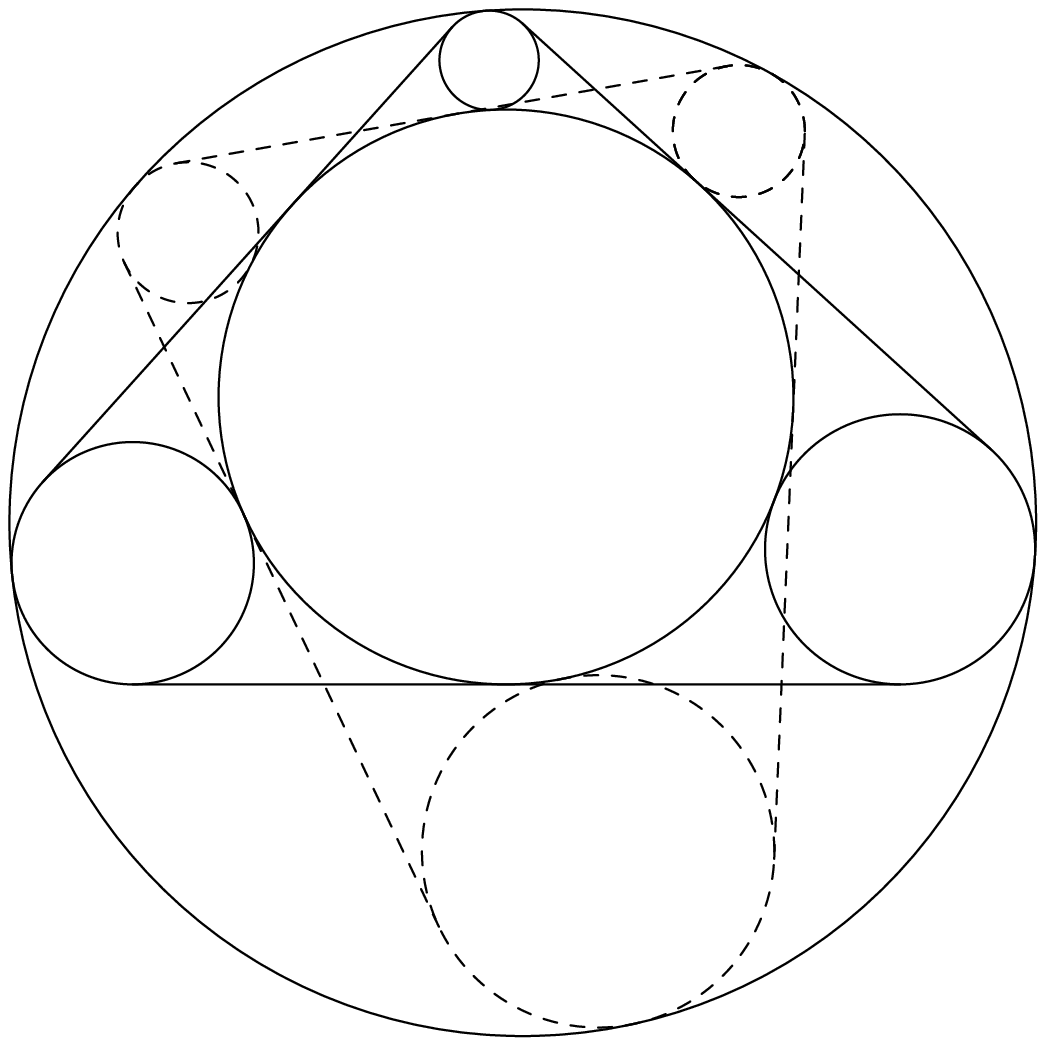}
\caption{}
\label{f2}
\end{center}\end{figure}

This theorem is strongly reminiscent of the classical closure theorems of Poncelet and Steiner (\cite{pon}, \cite{stein}). Poncelet's theorem deals solely with chains of tangents, and Steiner's theorem deals solely with chains of circles. Hence the name suggested by the author.

We proceed to show how Theorem \ref{t3} can be reduced to Poncelet's theorem.

\begin{theorem}
\label{t4}

Given are an ellipse $\e$ and a parabola $p$ of a common focus $F$. Then, when $p$ rotates about $F$, the common chord of $\e$ and $p$ (for all positions of $p$ for which it exists) is tangent to some fixed conic $\gamma$.

\end{theorem}

\begin{proof}

Let $\alpha$ be the plane in which $\e$ and $p$ lie. Let $s$ be a sphere which is tangent to $\alpha$ at $F$. For any point $X$ in space outside $s$, let $t_X$ be the length of the tangent from $X$ to $s$, $c_X$ be the cone formed by the lines through $X$ which are tangent to $s$, $\pi_X$ be the plane which contains the contact circle of $c_X$ with $s$, and $l_X$ be the conic section in which $c_X$ intersects $\alpha$.

Let $E$ and $P$ be such that $l_E = \e$ and $l_P = p$. For all $s$ whose radius is small enough, $E$ and $P$ will both be finite points and the common chord of $\e$ and $p$ will lie in the opposite half-space with respect to $\pi_E$ and $\pi_P$ as $E$ and $P$, respectively (since the circle $c_X \cap s$ will separate $X$ and $l_X$ on the surface of $c_X$ for $X = E$ and $X = P$).

Let $d(X, \chi)$ denote the signed distance from the point $X$ to the plane $\chi$. Then, for all points $X$ in $\e$, $d(X, \pi_E) : t_X$ is some real constant $c_1$ and for all points $X$ in $p$, $d(X, \pi_P) : t_X$ is some, possibly different, real constant $c_2$. It follows, then, that the common chord of $\e$ and $p$ lies in the plane $\pi$ which is the locus of all points $X$ such that $d(X, \pi_E) : d(X, \pi_P) = c_1 : c_2$. Notice that $\pi$ also contains the straight line $m = \pi_E \cap \pi_P$, and that, when $p$ varies, $P$ and $\pi$ vary also, but $c_1$, $c_2$, and $c_1 : c_2$ remain fixed.

Let $l$ be the line through $F$ which is perpendicular to $\alpha$, and let $A = \pi_E \cap l$, $B = \pi_P \cap l$, and $C = \pi \cap l$. Then $A$ and $B$ remain fixed when $p$ and $P$ vary (since $p$ varies by rotating about $F$), $\pi_E$ and $\pi_P$ intersect $l$ at fixed angles, and the ratio $d(C, \pi_E) : d(C, \pi_P) = c_1 : c_2$ remains constant. It follows that the ratio of the signed distances $|CA| : |CB|$ remains constant also and that $C$ remains constant, too.

When $p$, $P$, and $\pi_P$ vary, the planes $\pi_P$ envelop some straight circular cone $d$. The intersection $d \cap \pi_E$ is then some fixed conic section $\beta$ tangent to the line $m$. Let $\gamma$ be the projection of $\beta$ through the fixed point $C$ onto $\alpha$. As we saw above, the projection of $m$ through $C$ onto $\alpha$ is a straight line which contains the common chord of $\e$ and $p$. Since projection preserves tangency, this common chord will always be tangent to the fixed conic $\gamma$, as needed. \qedhere

\end{proof}

\begin{proof}[Proof of Theorem \ref{t3}]

Let $\o'$ of center $O'$ and radius $r'$ be any circle inscribed in the annulus between $k$ and $\o$ of centers $O$ and $I$ and radii $R$ and $r$. Then $OO' + O'I = R - r' + r' + r = R + r$ remains constant when $\o'$ varies, implying that $O'$ traces some fixed ellipse $\e$ of foci $O$ and $I$.

Let $t$ be any tangent to $\o$ and let $\o'$ and $\o''$ of centers $O'$ and $O''$ and radii $r'$ and $r''$ be the two circles inscribed in the annulus between $k$ and $\o$, tangent to $t$, and lying on the same side of $t$ as $\o$. Let $t'$ be the image of $t$ under homothety of center $I$ and coefficient $2$, and let $d(X, l)$ be the signed distance from the point $X$ to the line $l$. Then $O'I = r' + r = d(O', t')$ and $O''I = r'' + r = d(O'', t')$, implying that $O'$ and $O''$ both lie on the parabola $p$ of focus $I$ and directrix $t'$.

When $t$ varies, $t'$, and, consequently, $p$, rotates about $I$. Since $O'O''$ is the common chord of $\e$ and $p$, by Theorem \ref{t4} $O'O''$ remains tangent to some fixed conic $\gamma$.

We are only left to apply Poncelet's theorem to the two conics $\e$ and $\gamma$ and the chain $O_1$, $O_2$, ... $O_{n + 1}$ of the centers of the circles $\o_1$, $\o_2$, ... $\o_{n + 1}$. \qedhere

\end{proof}

\section{More on Conic Sections}

Theorem \ref{t4} remains true, its proof unaltered, when $p$ is taken to be an ellipse rather than a parabola. It also remains true when $p$ is an arbitrary conic section, provided that the common chord of $\e$ and $p$ is chosen ``just right'' and varies continuously when $p$ varies.

\begin{theorem}
\label{t5}

In the configuration of the generalized Theorem \ref{t4}, $F$ is a focus of $\gamma$.

\end{theorem}

\begin{lemma}

Let $c$ be a straight circular cone of axis $l$ and base the circle $k$ of center $O$ which lies in the plane $\alpha$. Let the plane $\beta$ intersect $c$ in the conic section $a$, $P$ be an arbitrary point on $l$, and $b$ be the projection of $a$ through $P$ onto $\alpha$. Then $O$ is a focus of $b$.

\end{lemma}

\begin{proof}[Proof of the Lemma]

If $\alpha \parallel \beta$, then $b$ is a circle of center $O$ and we are done. Let $\alpha \not\,\parallel \beta$ and $s = \alpha \cap \beta$.

Let $A$ be the vertex of $c$, $X$ be any point on $b$, $Y = PX \cap a$, $Z = AY \cap k$, and $U$ be the orthogonal projection of $Y$ onto $\alpha$. Notice that $A$, $O$, $P$, $X$, $Y$, $Z$, and $U$ all lie in some plane perpendicular to $\alpha$ and that $O$, $X$, $Z$, and $U$ are collinear.

It suffices to show that $d(X, s)$ is a linear function of the length of the segment $OX$ when $X$ traces $b$. For, suppose that we manage to prove that $d(X, s) = m \cdot OX + n$ where $m$ and $n$ are some real constants. Without loss of generality, $m \ge 0$ (otherwise, we can change the direction of $s$). Let $s'$ be the line which is parallel to $s$ and at a signed distance of $n$ from $s$. Then $b$, the locus of $X$, is the conic section of focus $O$, directrix $s'$, and eccentricity $m$.

For simplicity, we assume that $a$ and $k$ do not intersect and that $P$ lies between the intersection points of $\alpha$ and $\beta$ with $l$. In this case, $P$ lies between $A$ and $O$ and between $X$ and $Y$, $Y$ lies between $A$ and $Z$, $U$ lies between $O$ and $Z$, and $O$ lies between $X$ and $Z$.

We have \[ d(X, s) = \frac{ZX}{ZO} \cdot d(O, s) - \frac{OX}{ZO} \cdot d(Z, s) = \frac{R + OX}{R} \cdot d(O, s) - \frac{OX}{R} \cdot d(Z, s), \] where $R$ is the radius of $k$.

We can forget about the coefficient $\frac{1}{R}$. Since $d(O, s)$ is constant, we can also forget about the term $(R + OX) \cdot d(O, s)$. This leaves us to prove that \[ OX \cdot d(Z, s) \] is linear in $OX$.

We have \[ d(Z, s) = \frac{OZ}{OU} \cdot d(U, s) - \frac{ZU}{UO} \cdot d(O, s) = \frac{R}{OU} \cdot d(U, s) - \frac{ZU}{UO} \cdot d(O, s). \]

Since the shape of $\t YZU \sim \t AZO$ remains fixed when $X$ varies, the ratio $YU : ZU$ remains fixed as well. Since $Y$ varies in the fixed plane $\beta$, the ratio $d(U, s) : YU$ remains fixed also. It follows, then, that $d(U, s) : ZU$ remains fixed and that \[ d(Z, s) = c_1 \cdot \frac{ZU}{UO} \] for some real constant $c_1$.

Notice that $ZU : UO = ZY : YA$. By Menelaus' theorem, then, we have \[ \frac{ZY}{YA} \cdot \frac{AP}{PO} \cdot \frac{OX}{XZ} = 1 \Leftrightarrow \frac{ZY}{YA} = c_2 \cdot \frac{XZ}{OX}, \] where $c_2$ is some real constant.

From this, we obtain \[ OX \cdot d(Z, s) = OX \cdot c_1 \cdot c_2 \cdot \frac{XZ}{OX} = c_1 \cdot c_2 \cdot (R + OX), \] which, of course, is linear in $OX$, as needed. \qedhere

\end{proof}

\begin{proof}[Proof of Theorem \ref{t5}]

In the notation of the proof of Theorem \ref{t4}, apply the lemma to the cone $d$, the planes $\alpha$ and $\pi_E$, and the projection point $C$. \qedhere

\end{proof}

\begin{figure}[ht!]\begin{center}
\includegraphics[width=0.75\textwidth]{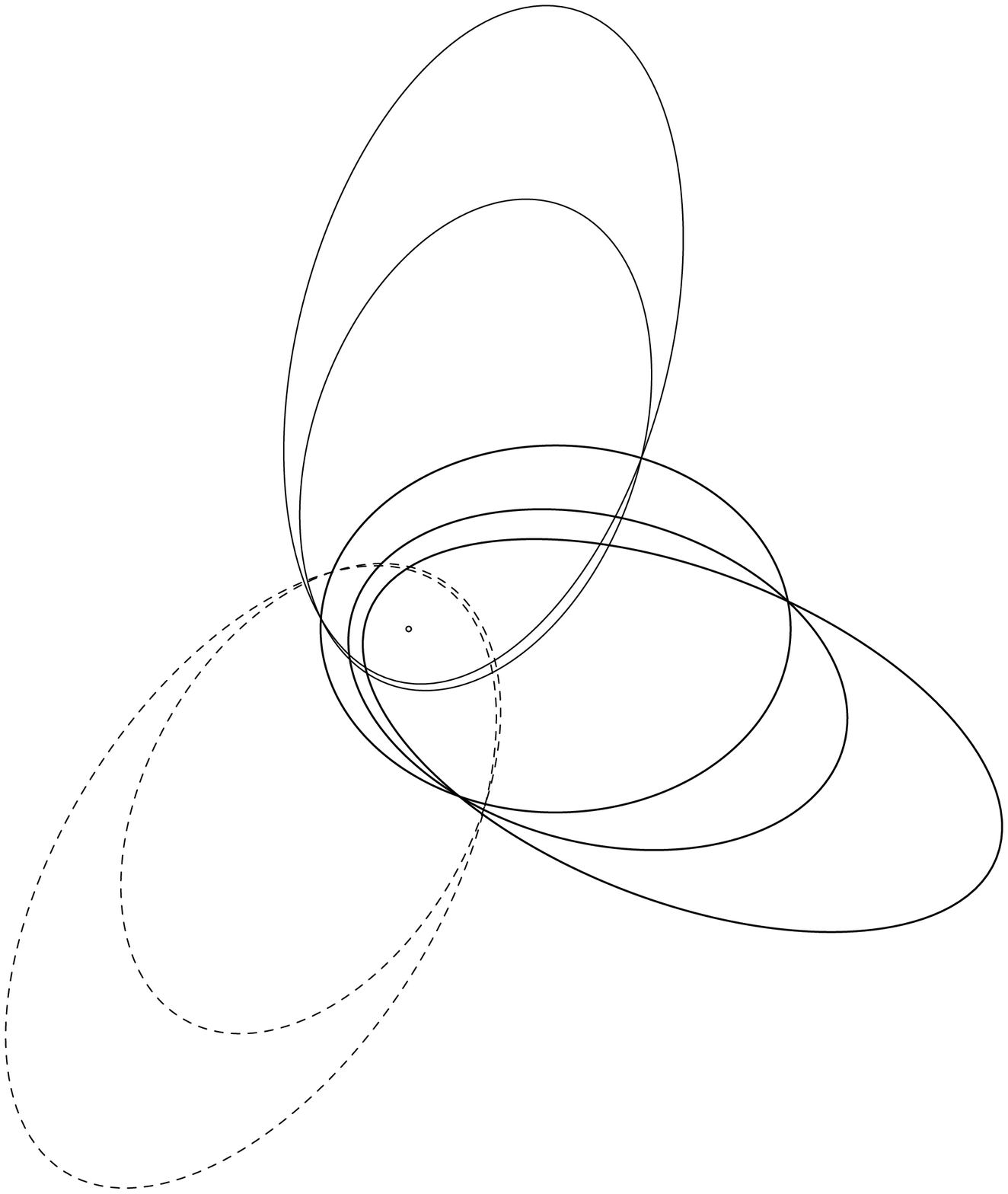}
\caption{}
\label{f3}
\end{center}\end{figure}

Theorem \ref{t5} has the following pleasing corollary.

\begin{theorem}
\label{t6}

Three ellipses $\e_1$, $\e_2$, and $\e_3$ of a common focus $F$ intersect in two distinct points. The ellipses $\e_1$ and $\e_2$ start to revolve continuously about $F$ so that they still intersect in two distinct points. Then the ellipse $\e_3$ can also revolve continuously about $F$ so that, at all times, all three ellipses continue to intersect in two distinct points. (Fig.\ \ref{f3})

\end{theorem}

\section{Open Questions}

Let $w = a_1a_2 \ldots a_n$ be any $n$-letter word over the two-letter alphabet $\{c, s\}$. For any annulus $a$ delimited by an inner circle $\o$ and an outer circle $k$, we construct a chain $u_1$, $u_2$, ... $u_{n + 1}$ of circles and segments as follows:

(a) $u_i$ is a circle inscribed in $a$ if $a_i = c$ and a segment tangent to $\o$ whose endpoints lie on $k$ if $a_i = s$ (for $i = n + 1$, we set $a_{n + 1} = a_1$);

(b) If at least one of $u_i$ and $u_{i + 1}$ is a circle, then they are tangent;

(c) If $u_i$ and $u_{i + 1}$ are both segments, then they have a common endpoint;

(d) For $1 < i \le n$, the contact points of $u_i$ with $u_{i - 1}$ and $u_{i + 1}$ are separated on $u_i$ by the contact points of $u_i$ with $\o$ and $k$ if $u_i$ is a circle, and by the contact point of $u_i$ with $\o$ if $u_i$ is a segment.

If, for any annulus $a$, it is true that if $u_1 \equiv u_{n + 1}$ for some initial $u_1$ then $u_1 \equiv u_{n + 1}$ for all initial $u_1$, and there exists at least one annulus $a$ for which the premise of this implication is not void, then we say that the word $w$ is a \emph{closure sequence}.

Poncelet's theorem tells us that $s^n$ is a closure sequence for all $n \ge 3$. Steiner's theorem tells us that $c^n$ is a closure sequence for all $n \ge 3$. Finally, the Mixed Poncelet-Steiner theorem tells us that $(cs)^n$ and $(sc)^n$ are closure sequences for all $n \ge 2$.

Are there any other closure sequences? Given a particular closure sequence $w$, what relation must the radii of $\o$ and $k$ and their intercenter distance satisfy in order for the chain described by $w$ to close? Is this relation always of the form $f(R, r, d) = 0$ for some rational function $f$? Finally, is it true that, for any closure sequence $w$ and any positive integer $n$, $w^n$ is also a closure sequence?

\end{document}